\documentclass{article}
\usepackage{amssymb}
\usepackage{amsthm}
\usepackage{hyperref}
\usepackage{amsmath}
\usepackage{amsfonts}
\usepackage{latexsym} 
\usepackage{indentfirst}
\usepackage{graphicx}
\usepackage[margin=1.0in]{geometry}
\usepackage{tikz}
\usepackage{verbatim}
\usepackage{tikz-3dplot}
\usepackage{longtable}
\usepackage{float}

\usepackage{makeidx}

\makeindex

\newtheorem{theorem}{Theorem}[section]
\newtheorem{proposition}[theorem]{Proposition}

\newtheorem{question}[theorem]{Question}
\newtheorem{lemma}[theorem]{Lemma}

\theoremstyle{definition}
\newtheorem{definition}[theorem]{Definition}

\def\Z{{\mathbb Z}} 
\def\N{{\mathbb N}} 
\def\cal{\mathcal}
\newcommand\tab[1][1cm]{\hspace*{0.6cm}}
\usepackage[utf8]{inputenc}

\theoremstyle{remark}
\newtheorem{rem}[theorem]{Remark}

\title{On Induced Subgraphs of the Hamming Graph}
\author{Dingding Dong\footnote{University of Chicago. Email: ddong@math.harvard.edu.}}
\date{October 15, 2020}
\begin{document}
\maketitle


\begin{abstract} 
In connection with his solution of the Sensitivity Conjecture,
Hao Huang (arXiv: 1907.00847, 2019) asked the following question:
Given a graph $G$ with high symmetry, what can we say about
the smallest maximum degree of induced subgraphs of $G$ with
$\alpha(G)+1$ vertices, where $\alpha(G)$ denotes the size of
the largest independent set in $G$\,?  We study this question
for $H(n,k)$, the $n$-dimensional Hamming graph over an alphabet
of size $k$. Generalizing a construction by Chung et al. (JCT-A, 1988), we prove that $H(n,k)$ has an induced subgraph with more than
$\alpha(H(n,k))$ vertices and maximum degree at most 
$\lceil \sqrt{n}\rceil$. Chung et al. proved this statement for 
$k=2$ (the $n$-dimensional cube). 
\end{abstract}


\section{Introduction}

For a graph $G=(V,E)$, let $\alpha(G)$ denote the independence number of $G$ (the maximum size of an independent set). In this paper we study the quantity $f(G)$ defined as the smallest maximum degree of induced subgraphs of $G$ with $\alpha(G)+1$ vertices. The Hamming graph $H(n,k)$ is a graph on the vertex set 
$\Sigma^n$  
where 
$\Sigma$   
is an alphabet of size $k$, such that two vertices are adjacent if and only if they differ in precisely one coordinate.

Notice that $H(n,2)=Q^n$ is the $n$-dimensional cube. The study of $f(Q^n)$ goes back to a 1988 paper by Chung, F\"uredi, Graham and Seymour \cite{Chung}, who\footnote{In this paper, $\log$ refers to base-2 logarithms.} proved $\frac{1}{2}(\log n-\log\log n+1)<f(Q^n)\leq \lceil\sqrt{n}\rceil$. Their proof of the lower bound used the connection between $f(Q^n)$ and the sensitivity of Boolean functions. Gotsman and
Linial \cite{Gotsman} made a significant further step in formalizing this connection. They showed that the inequality $f(Q^n)\geq n^c$ for some constant $c>0$ is equivalent to the Sensitivity Conjecture for Boolean functions, proposed by Nisan and Szegedy \cite{Nisan}. (See the survey by Hatami et al. \cite{Hatami} on the Sensitivity Conjecture.)

Hao Huang \cite{Huang} recently proved that $f(Q^n)\geq \lceil\sqrt{n}\rceil$. This lower bound, according to the result by Gotsman and Linial, confirmed the Sensitivity Conjecture for Boolean functions. In the last section of his paper, Huang asks the following question.

\begin{question}[Huang]  \label{huang-question} 
What can we say about $f(G)$ for graphs $G$ with high symmetry?
\end{question}

We study this question for the Hamming graphs. Generalizing the proof of the inequality $f(Q^n)\leq \lceil\sqrt{n}\rceil$ by Chung et al. \cite{Chung}, in this note we prove the following bound.

\begin{theorem}  \label{thm:main}  
For all $k,n\geq 1$, we have $f(H(n,k))\leq\lceil \sqrt{n}\rceil$. In fact, $H(n,k)$ has a bipartite induced subgraph with maximum degree at most $\lceil \sqrt{n}\rceil$ and more than $\alpha(H(n,k))$ vertices.
\end{theorem}    


\section{Preliminaries and notation}

For a graph $G=(V,E)$, define $\Delta(G)$ to be the maximum degree of $G$. For a subset $W\subseteq V$, define $G[W]$ to be the induced subgraph of $G$ on vertex set $W$. It will be convenient to take $\Z_k=\Z/k\Z$ as the alphabet for the Hamming graph, so the set of vertices of $H(n,k)$ is $\Z_k^n$\,. We view the elements $v\in\Z_k^n$ as functions from $[n]=\{1,\dots,n\}$ to $\Z_k$ and set $v=(v(1),\dots,v(n))$. It is not hard to see that for all $k,n\geq 1$, we have $\alpha(H(n,k))=k^{n-1}$. 


\section{Constructing a family of induced subgraphs}\label{sec:construction}

In this section, we construct a family $\cal A$ of $k^2$ induced bipartite subgraphs of $H(n,k)$. In Section~\ref{sec:proof} we show that (a) each member of $\cal A$ is bipartite (Proposition~\ref{prop:bipartite}), (b) each member of $\cal A$ has maximum degree at most $\lceil \sqrt{n}\rceil$ (Proposition~\ref{prop:degreebound}), and (c) at least one member of $\cal A$ has more than $\alpha(H(n,k))$ vertices (Proposition~\ref{prop:large}). Theorem~\ref{thm:main} is an immediate consequence of these statements.

Following Chung et al. \cite{Chung}, we note that for all $n\in\N$, there exists a partition $[n]=F_1\overset{\cdot}{\cup}\dots \overset{\cdot}{\cup} F_q$ such that $|q-\sqrt{n}|<1$ and each $||F_j|-\sqrt{n}|<1$. We fix such a partition for the rest of this paper and use it to define partitions of $\Z_k^n$, the set of vertices.

\begin{definition}   \label{def:partition1}
We partition the vertex set of $H(n,k)$ as
$$\Z_k^n=X\overset{\cdot}{\cup} Y$$
where
\begin{equation}
X:=\left\{v\in \Z_k^{n}: F_j\subseteq v^{-1}(0)\text{ for some }j\in[q]\right\},
\end{equation}
\begin{equation}Y:=\left\{v\in \Z_k^{n}: F_j\not\subseteq v^{-1}(0)\text{ for all }j\in[q]\right\}.
\end{equation}
\end{definition}

\begin{definition}   \label{def:partition2}
We further partition $X$ and $Y$ as
$$X=X_0\overset{\cdot}{\cup}X_1\overset{\cdot}{\cup}\dots\overset{\cdot}{\cup}X_{k-1},$$
$$Y=Y_0\overset{\cdot}{\cup}Y_1\overset{\cdot}{\cup}\dots\overset{\cdot}{\cup}Y_{k-1}$$
where
\begin{equation}
X_i:=X\cap \left\{v\in\Z_k^n:\sum_{\ell=1}^nv(\ell)=i\right\},
\end{equation}
\begin{equation}
Y_i:=Y\cap \left\{v\in\Z_k^n:\sum_{\ell=1}^nv(\ell)=i\right\}.
\end{equation}
\end{definition}

We consider the following family of $k^2$ induced subgraphs of $H(n,k)$.
\begin{equation}
\mathcal{A}:=\left\{H(n,k)[X_{i_1}\cup Y_{i_2}]:i_1,i_2\in\Z_k\right\}.
\end{equation}


\section{Analysis of the construction}\label{sec:proof}

In this section, we prove properties (a), (b) and (c) of the family $\cal A$ indicated in the first paragraph of Section~\ref{sec:construction}. We first address (a).

\begin{proposition}~\label{prop:bipartite}
Each of the sets $X_i$, $Y_i$ $(i\in\Z_k)$ is independent in $H(n,k)$. In particular, for all $i_1,i_2\in\Z_k$, the induced subgraph $H(n,k)[X_{i_1}\cup Y_{i_2}]$ is bipartite.
\end{proposition}
\begin{proof}This follows from the fact that the coordinates of vertices in each of the sets $X_{i}$, $Y_i$ have the same sum.
\end{proof}


\subsection{Maximum degree bound}
We prove that for all $i_1,i_2\in\Z_k$, the maximum degree of $H(n,k)[X_{i_1}\cup Y_{i_2}]$ is at most $\lceil \sqrt{n}\rceil$. First notice that if $v\in X$ has a neighbor in $Y$, then $v^{-1}(0)$ contains exactly one of $F_1,\dots,F_q\,$. By contradiction, if $v^{-1}(0)$ contains more than one of $F_1,\dots,F_q\,$, then any neighbor of $v$ still contains some $F_j$ and therefore cannot be in $Y$. This allows us to make the following definition. 

\begin{definition}
For any $v\in X$ with a neighbor in $Y$, define $j(v)$ to be the unique $j\in[q]$ such that $F_j\subseteq v^{-1}(0)$.
\end{definition}

\begin{proposition}    \label{prop:tech2}  
 Fix $v_2\in Y$. If $v_1,v_1'\in X$ are neighbors of $v_2$ and $j(v_1)=j(v_1')$, then $v_1=v_1'\,$. In particular, each vertex in $Y$ has at most $q$ neighbors in $X$.
\end{proposition}  
\begin{proof}Suppose $v_1\in X$ is a neighbor of $v_2$. Since $F_{j(v_1)}\subseteq v_1^{-1}(0)$ and $F_{j(v_1)}\not\subseteq v_2^{-1}(0)$, there exists some $\ell_0\in F_{j(v_1)}$ such that $\left(v_2^{-1}(0)\cap F_{j(v_1)}\right)\overset{\cdot}{\cup}\{\ell_0\}=F_{j(v_1)}\,$. Hence $v_1$ is the only element of $$\left\{v\in\Z_k^n: v\sim v_2, \hspace{0.1cm}F_{j(v_1)}\subseteq v^{-1}(0)\right\},$$ which we obtain from $v_2$ by changing the image of $\ell_0$ from $v_2(\ell_0)$ to 0.
\end{proof}

\begin{proposition}    \label{prop:tech3} 
If $v_1\in X_{i_1}$ and $v_2\in Y_{i_2}$ are neighbors, then there 
exists $\ell_0\in F_{j(v_1)}$ such that the following hold.
\begin{enumerate} 
    \item $v_1(\ell_0)=0$;
    \item $v_2(\ell_0)=i_2-i_1\,$;
    \item $v_1(\ell)=v_2(\ell)$ for all $\ell\in[n]\setminus\{\ell_0\}$.
\end{enumerate}
In particular, for all $i_1,i_2\in\Z_k$, each vertex in $X_{i_1}$ has at most $\displaystyle \max_{j\in[q]}|F_j|$ neighbors in $Y_{i_2}$.
\end{proposition}
\begin{proof}
Items 1 and 3 follow from the fact that $v_1\in X$, $v_2\in Y$ and $v_1\sim v_2\,$. To see item 2, notice that $$i_2-i_1=\sum_{\ell=1}^n \left(v_2(\ell)-v_1(\ell)\right)=v_2(\ell_0)-0.$$
\end{proof}

\begin{proposition}   \label{prop:degreebound}  
For all $i_1,i_2\in\Z_k$,
$$\Delta(H(n,k)[X_{i_1}\cup Y_{i_2}])\leq \lceil\sqrt{n}\rceil.$$
\end{proposition}
\begin{proof}Combining Propositions~\ref{prop:bipartite},~\ref{prop:tech2} and~\ref{prop:tech3}, we obtain that
\begin{equation}
   \Delta(H(n,k)[X_{i_1}\cup Y_{i_2}])\leq \max\left\{\max_{j\in[q]}|F_j|,q\right\}\leq \lceil\sqrt{n}\rceil.
\end{equation}
\end{proof}


\subsection{Size greater than the independence number}

We now prove that there exist $i_1,i_2\in\Z_k$ such that the induced subgraph $H(n,k)[X_{i_1}\cup Y_{i_2}]$ has size greater than $ \alpha(H(n,k))$.

\begin{lemma}\label{lem:congruence}\quad
$|X|\equiv (-1)^{q+1}\pmod k$. In particular, $|X|$ is not divisible by $k$.
\end{lemma}
\begin{proof}The proof generalizes the inclusion-exclusion argument of Chung et al. \cite{Chung}. By the inclusion-exclusion principle, we have
$$|X|=\sum_{j\in[q]}k^{n-|F_j|}-\sum_{\substack{j_1,j_2\in[q]\\j_1< j_2}}k^{n-|F_{j_1}\cup F_{j_2}|}+\dots+(-1)^q\sum_{\substack{j_1,\dots,j_{q-1}\in[q]\\j_1<\dots<j_{q-1}}}k^{n-|F_{j_1}\cup\dots\cup F_{j_{q-1}}|}+(-1)^{q+1}.$$
Since all terms but the last are divisible by $k$, the statement follows.
\end{proof}

\begin{proposition}   \label{prop:large}  
There exist $i_1,i_2\in\Z_k$ such that $H(n,k)[X_{i_1}\cup Y_{i_2}]$ has size greater than $ \alpha(H(n,k))$.
\end{proposition}
\begin{proof}
Choose $i_1, i_2\in\Z_k$ such that $\displaystyle |X_{i_1}|=\max_{i\in\Z_k}|X_i|$ and $\displaystyle |Y_{i_2}|=\max_{i\in\Z_k}|Y_i|$. By Lemma~\ref{lem:congruence}, it follows that $$ |X_{i_1}|>\frac{1}{k}\sum_{i\in\Z_k}|X_{i}|.$$
Since also
$$|Y_{i_2}|\geq\frac{1}{k}\sum_{i\in\Z_k}|Y_{i}|,$$
we have
$$|X_{i_1}\cup Y_{i_2}|=|X_{i_1}|+|Y_{i_2}|>\frac{1}{k}\left (\sum_{i\in\Z_k}|X_{i}|+\sum_{i\in\Z_k}|Y_{i}|\right)=\frac{1}{k}(|X|+|Y|)=\frac{1}{k}|\Z_k^n|=k^{n-1}=\alpha(H(n,k)).$$
\end{proof}

\begin{proof}[Proof of Theorem~\ref{thm:main}] 
Combine Propositions~\ref{prop:degreebound} and~\ref{prop:large}.
\end{proof}

\begin{rem}Readers familiar with Chung et al. \cite{Chung} will notice that our proof closely follows the steps of that paper. The challenge for us was to find the right construction of the family $\cal A$ of induced subgraphs that permits an extension of the analysis given in \cite{Chung}. We were surprised to find that the resulting bound, $\lceil\sqrt{n}\rceil$, does not depend on $k$.

\end{rem}


\section{Open questions on the lower bound} 

It is natural to ask whether $\sqrt{n}$ is the true order of magnitude of  $f(H(n,k))$ for $k\geq 3$. Huang \cite{Huang} proved the lower bound $f(H(n,2))\geq\sqrt{n}$ using 
linear algebra. We were unable to generalize his argument to Hamming graphs 
with larger alphabets. Unfortunately we were unable even to obtain a lower bound of logarithmic order like the Chung et al. lower bound for $k=2$. We were able to generalize one of their lemmas toward this lower bound; this result might be of interest in its own right.

\begin{lemma}  \label{lem:partial}  
If $S=(V(S),E(S))$ is a subgraph of $H(n,k)$ with average degree $\overline{d}$, then $\log|V(S)|\geq \frac{\overline{d}}{k-1}$.
\end{lemma}
\begin{proof}The statement is clear for $n=1$. Inductively, split $\Z_k^n$ into $k$ copies of $\Z_k^{n-1}$: $$\Z_k^n=Z_1\overset{\cdot}{\cup}\dots \overset{\cdot}{\cup} Z_k,$$$$Z_i=\{v\in\Z_k^n:v(1)=i\}.$$ For every $i\in[k]$ set $V_i=V(S)\cap Z_i$. By throwing away the empty $V_i$ and permuting the indices, we may assume $|V_1|\leq \dots\leq |V_{K} |$ with all $V_i$ nonempty.

Let $s_1$ be the number of edges between $V_{1}$ and $\bigcup_{j=2}^{K}V_j$ in $G$, $s_2$ the number of edges between $V_2$ and $\bigcup_{j=3}^{K}V_j$ in $G$, $\dots$, $s_{K-1}$ the number of edges between $V_{K-1}$ and $V_{K}$ in $G$. This gives $|V_i|\geq \frac{s_i}{k-1}$ for all $i\in[K-1]$.

The inductive hypothesis gives
$$|V_i|\log|V_i|\geq  \frac{1}{k-1} \left(\sum_{v\in V_i}\deg_{G[V_i]}(v)\right),$$
so that
\begin{equation}
    \sum_{i=1}^{K}|V_i|\log|V_i|+\frac{2}{k-1}\left (\sum_{i=1}^{K-1}s_i\right )\geq \frac{1}{k-1}\left(\sum_{v\in V}\deg_{G}(v)\right).
\end{equation}
Notice that for $0< p\leq q$, we always have
\begin{equation}
(p+q)\log(p+q)\geq p\log p+q\log q+2p.
\end{equation}
Combining (7) and (8) gives
$$(|V_1|+\dots+|V_K|)\log(|V_1|+\dots+|V_K|)\geq |V_{1}|\log|V_{1}|+(|V_{2}|+\dots+|V_K|)\log(|V_{2}|+\dots+|V_K|)+2|V_1|$$
$$\geq \dots\geq \sum_{i=1}^K|V_i|\log|V_i|+2\sum_{i=1}^{K-1}|V_i|\geq \sum_{i=1}^K|V_i|\log|V_i|+\frac{2}{k-1}\left (\sum_{i=1}^{K-1}s_i\right)\geq \frac{1}{k-1} \left(\sum_{v\in V_i}\deg_{G[V_i]}(v)\right).$$
\end{proof}

\bigskip\tab\\   
\textbf{Acknowledgements.} This work was done while the author was an undergraduate at the University of Chicago.  I wish to thank Professor L\'aszl\'o Babai for suggesting that I work on this problem and discussing it with me.


\end{document}